\documentclass[11pt]{article}
\usepackage[dvips]{graphicx}
\usepackage{amsmath,amsfonts,amssymb}
\usepackage{amsthm}

\def\R{\mathbb{R}}

\providecommand{\norm}[1]{\lVert#1\rVert}

\newcommand{\bigpar}{\par\quad\newline\noindent}
\newcommand{\Dconst}{k_0}
\newcommand{\ini}{\mathrm{ini}}
\newcommand{\eps}{\varepsilon}

\newtheorem{theorem}{Theorem}[]

\newtheorem{lemma}[theorem]{Lemma}
\newtheorem{proposition}[theorem]{Proposition}

\begin{document}

\title{%
On a Model for Mass Aggregation with Maximal Size
}%
\date{\today}%
\author{Ondrej Bud\'a\v{c}\thanks{ %
        Vrije Universiteit Amsterdam,
        ondrob@gmail.com}
    \;
    Michael Herrmann\thanks{ %
        Oxford Centre for Nonlinear PDE (OxPDE),
        michael.herrmann@maths.ox.ac.uk}
    \;
    Barbara Niethammer\thanks{ %
        Oxford Centre for Nonlinear PDE (OxPDE),
        niethammer@maths.ox.ac.uk}
    \;
        Andrej Spielmann\thanks{ %
        \'{E}cole polytechnique f\'{e}d\'{e}rale de Lausanne,
        andrej.spielmann@epfl.ch}
}
\maketitle
\begin{abstract}
We study a kinetic mean-field equation for a  system of particles with different sizes, in
which particles are allowed to coagulate only if their sizes sum up to a prescribed
time-dependent value. We prove well-posedness of this model, study the existence of
self-similar solutions, and analyze the large-time behavior mostly by numerical simulations.
Depending on the parameter $\Dconst$, which controls the probability of coagulation,
we observe two different scenarios: For $\Dconst>2$ there exist two self-similar solutions to
the mean field equation, of which one is unstable. In numerical simulations we observe that
for all initial data the rescaled solutions converge to the stable self-similar solution.
For $\Dconst<2$, however, no self-similar behavior occurs as the solutions converge in the
original variables to a limit that depends strongly on the initial data. We prove rigorously a
corresponding statement for $\Dconst\in (0,1/3)$. Simulations for the cross-over case
$\Dconst=2$ are not completely conclusive, but indicate that, depending on the initial data,
part of the mass evolves in a self-similar fashion whereas another part of the mass remains in
the small particles.
\end{abstract}
%
%
%
\quad\newline\noindent%
\begin{minipage}[t]{0.15\textwidth}%
{\small Keywords:} %
\end{minipage}%
\begin{minipage}[t]{0.8\textwidth}%
\small %
\emph{aggregation with maximal size}, %
\emph{self-similar solutions}, \\%
\emph{coarsening in coagulation models} %
\end{minipage}%
\medskip
\newline\noindent
\begin{minipage}[t]{0.15\textwidth}%
\small MSC (2000): %
\end{minipage}%
\begin{minipage}[t]{0.8\textwidth}%
\small %
45K05, 	
82C22  
\end{minipage}%
%
%
%

\section{Introduction}
%
%
Mass aggregation is a fundamental process that appears in a large range of applications, such
as formation of aerosols, polymerization, clustering of stars or ballistic aggregation, see
for instance \cite{Drake,Fried,Ziff}. In all these applications, certain types of 'particles'
form  clusters that are characterized by their 'size' or 'mass' $x$. 
 Smoluchowski's \cite{Smolu} classical mesoscopic mean-field description of
irreversible aggregation processes describes the evolution of the number density  $g(t,x)$ of
clusters of size $x$ per unit volume at time $t$. Clusters of size $x$ and $y$ can coalesce
by binary collisions to form clusters of size $x+y$ at a rate given by a kernel
$K(x,y)$, such that the dynamics of $g$ is given by
\begin{equation}
\label{coag1}
\frac{\partial}{\partial t} g(t,x) = \frac 1 2  \int_0^{x}
K(y,x-y) g(t,x-y) g(t,y) \,dy \, -\, g(t,x) \int_0^{\infty}
K(x,y) g(t,y)\,dy\,.
\end{equation}
An issue of fundamental interest in the mathematical analysis of coagulation processes is the
phenomenon of dynamic scaling for homogeneous kernels. This means that for initial data from
a certain class the solution to \eqref{coag1} converges to a certain self-similar solution.
Unfortunately, except for some special kernels such as the constant one, this question is
still only poorly understood (see e.g. \cite{Leyvraz} for an overview).  While it has been
common belief in the applied literature that self-similar solutions are unique, it has
recently been shown for some special kernels \cite{MP1} that there is a whole one-parameter
family of self-similar solutions. These solutions can be distinguished by their tail
behavior, and their respective domains of attraction are characterized by the tail behavior
of the initial data. 
\par
In contrast to this, very little is known for other homogeneous kernels. The existence of
fast-decaying self-similar solutions for a range of kernels is established in \cite{EMR,FL1},
but both the existence of solutions with algebraic tail and the uniqueness of solutions 
are still open problems. In general, unless explicit methods such as the Laplace transform work,
just proving existence of self-similar solutions to a coagulation equation can be a formidable
task.
\par
Thus, despite their fundamental role, many properties of mean-field models for coagulation
processes,  in particular with respect to dynamic scaling, are not well understood. Motivated
by an application of elasto-capillary coalescence of wetted polyester lamellas \cite{BBR}, we
investigate this question for a special singular coagulation kernel $K$. This kernel allows
only those clusters to coagulate that can form clusters of a given maximal size. To our
knowledge  this is the first mathematical analysis for such type of kernels. Even though
tails of the size distribution do not play a role here, we find that self-similar solutions
are still not unique and the analysis of the long-time behavior turns out to be delicate.
\bigpar
In model considered here, two particles can coagulate only if the sum of their sizes is
equal
to $M(t)$, where $M$ is a given increasing function of time. This means that at time $t$ only
particles of size $M(t)$ are emerging and the amount of particles of all smaller sizes is
decreasing. This corresponds to $K(x,y,t)\sim\delta_0(x+y-M(t))$, where $\delta_0$ denotes the
delta-distribution in $0$ and the factor of proportionality may depend on time.
\par
As above, we denote the number density  of particles of size $x$ at time $t$ by  $g(t,x)$. The
total number of particles per unit volume at time $t$ is then $N(t)=\int_0^{M(t)}g(t,x)dx$.
At time $t$, the density of particles of any size $x<M(t)$ is decreasing at a rate
proportional to the probability density of a particle of size $x$ meeting a particle of size
$M(t)-x$. The coagulation process can hence be described by
\begin{align*}
\frac{\partial g}{\partial t}(t,x)=-K(t)\frac {g(t,x)}{N(t)}\frac {g(t,M(t)-x)}{N(t)}
\end{align*}
where $K(t)$ is a rate function proportional to the expected number of coagulation events per
unit time. Motivated by \cite{BBR} we make the ansatz $K(t)=\Dconst M'(t)N(t)$, where
$\Dconst$ is a constant of proportionality that depends on the particular physical process.
The above equation hence reads
\begin{equation}
\label{eq:gen0}%
\frac{\partial g}{\partial t}(t,x)=-\frac{\Dconst M'(t)}{N(t)}g(t,x)g(t,M(t)-x).
\end{equation}
The coagulation process described above neither creates nor destroys mass, which is expressed by
the mass conservation equation
\begin{equation}
\label{eq:gen0a}
\int_0^{M(t)}xg(t,x)dx=\sigma
\end{equation}
where $\sigma$ is a constant. From (\ref{eq:gen0}) and (\ref{eq:gen0a}) we will 
derive an equivalent condition for $g(t,M(t))$, see 
Equation (\ref{eq:gen3}) below.
\par
For $x>M(t)$, $g(t,x)$ is not changing since particles of size greater than $M(t)$ can not
form a particle of size $M(t)$ by coagulation. Hence $\frac{\partial g}{\partial t}(t,x)=0$
for $x>M(t)$. We are normally interested in processes where the value of $M$ at the starting
time is greater than the size of all particles existing initially, and hence we also assume
$g(t,x)=0$ for $x>M(t)$.
\par
An important feature of Equation (\ref{eq:gen0}) is its invariance under reparametrization
of time. As a consequence $M(t)$ is not determined by the initial data but can be
prescribed to be an arbitrary increasing function of time (see also \cite{GM,MNP}, where
such an invariance has been crucial in the analysis of a model for min-driven clustering).
In the following we choose
$M(t)=t$ and also normalize 
 the mass by setting $\sigma=1$. 
 Consequently, in what follows we study the system of
equations
\begin{equation}
\label{Eqn:Model}
\frac{\partial{g}}{\partial t}(t,x)=-\frac \Dconst { N(t)}g(t,x)g(t,t-x),\quad
N(t)=\int_0^t g(t,x)dx
,\quad
\int_0^tx g(t,x)dx=1
\end{equation}
with $t\geq1$ and $0\leq{x}\leq{t}$.
\par
The first aim of this article is to establish well-posedness of the initial value problem and
to study the long time behavior of solutions to \eqref{Eqn:Model}. To this end, we prove
global existence and uniqueness of mild solutions in Section
\ref{Sec:GEU}, Theorem \ref{Theo.Existence.Mild.Solutions} by rewriting \eqref{Eqn:Model} as a
fixed point equation. In contrast to coagulation equations with more
regular kernels, for which well-posedness can often be proved in the space of probability
measures \cite{FL3}, here we need to work with  functions that are continuous
up to the points $x=1$ and $x=t$. The reason that the
fixed-point argument is not quite straightforward is that at any time there is an influx of
particles of the largest size $M(t)$ that leads to a nontrivial boundary condition.
\par
In Section \ref{Sec:SSS} we study the existence of self-similar solutions.
Combining rigorous arguments with numerical simulations, we identify $\Dconst =2$ as a
critical value: For $\Dconst <2$ no self-similar solutions exist but for $\Dconst >2$ we find
two self-similar solutions which have different shape and stability properties. Section
\ref{Sec:LTB} is devoted to numerical simulations of initial value problems. Our results for
$\Dconst >2$ suggest that, after a suitable rescaling, each solution converges to the second
self-similar solution. For $\Dconst <2$, however, $g$ converges to a steady state
$g_{\infty}(x)$, $x\in\R^+$, whose shape  depends strongly on the initial data. We finally
prove this assertion in Proposition \ref{T.main} for $\Dconst<\tfrac{1}{3}$.
%
%
%
\section{Global existence and uniqueness of solutions}\label{Sec:GEU}
%
In this section we derive a notion of mild solutions to \eqref{Eqn:Model} that
relies on an appropritate reformulation of the mass constraint
$\int_0^t{x}g(t,x)dx=1$. We then use Banachs's Contraction Mapping Theorem to
prove the local existence and uniqueness of mild solutions, and finally employ some a priori
estimates to show that these mild solutions exist globally in time.

%
\subsection{Notion of mild solutions and main result}
%
%
%
%
To reformulate the mass constraint \eqref{Eqn:Model}$_3$ we suppose that a
piecewise smooth solution $g(t,x)$ to \eqref{Eqn:Model}$_1$ and \eqref{Eqn:Model}$_2$  is
given. We then have
\begin{eqnarray}
\label{Ass.Initial.Data}
 0&=&\frac d{dt}\int_0^txg(t,x)dx=\int_0^tx\frac{\partial g}{\partial t}(t,x)dx+tg(t,t),
\end{eqnarray}
and due to \eqref{Eqn:Model}$_1$ we obtain
\begin{equation}
 tg(t,t)=\frac {\Dconst}{N(t)}\int_0^txg(t,x)g(t,t-x)dx.\label{eq:btw1}
\end{equation}
Moreover, substituting $x\rightsquigarrow{x-t}$ we find
\begin{eqnarray*}
\int_0^txg(t,x)x(t,t-x)dx&=&t\int_0^tg(t,t-x)g(t,x)dx
-\int_0^txg(t,t-x)g(t,x)dx,
\end{eqnarray*}
so (\ref{eq:btw1}) implies
\begin{equation}
 g(t,t)=\frac {\Dconst}{2N(t)}\int_0^tg(t,x)g(t,t-x)dx.\label{eq:gen3}
\end{equation}
This condition is equivalent to \eqref{Eqn:Model}$_3$ provided that $g(t,x)$ solves
\eqref{Eqn:Model}$_1$, and for the ease of notation we introduce the operator
$B:g\mapsto{B}[g]$ by
\begin{eqnarray*}
B[g](t)&:=&\frac{\Dconst}{2N(t)}\int_0^tg(t,x)g(t,t-x)dx,
\end{eqnarray*}
where $N$ depends on $g$ via \eqref{Eqn:Model}$_2$.
\par
In what follows we fix some $T>0$ and introduce $X$ as the space of all functions $g$ that
satisfy
\begin{enumerate}
 \item[i)] $g:\Omega=[1,T]\times[0,T]\rightarrow\R_0^+$,        
 \item[ii)] $g(1,x)=g_\ini(x)$ for all $x\in[0,1]$,
 \item[iii)] $g(t,x)=0$ if $x>t$,
 \item[iv)] $g$ is continuous in all points $(t,x)$ with $x\neq 1$ and $x\neq{t}$,
\end{enumerate}
where the initial data $g_\ini$ are supposed to satisfy
\begin{align*}
g_\ini\in{C}([0,1];\,\R^+),\qquad\int_0^1x{g_\ini(x)}dx=1.
\end{align*}
To introduce the  notion of mild solutions we now define the operator
$\Gamma:X\rightarrow X$ by
\begin{equation}
\label{Eqn:MildSolution}
 \Gamma[g](t,x) = \left\{\begin{aligned}
  &g_\ini(x)\exp\left({-{\Dconst}\int_1^t\frac{g(s,s-x)}{N(s)}ds}\right)\quad && \text{
for } x<1,\\
  &B[g](x)\exp\left({-{\Dconst}\int_x^t\frac{g(s,s-x)}{
N(s)}ds}\right)&& \text{ for }1\leq
 x\leq t,\\
  &0&& \text{ for } t<x,
\end{aligned}
\right.
\end{equation}
with  $N(t)=\int_0^t{}g(t,x)dx$. Notice that $\Gamma$ maps $X$ into itself, and that for each
$\tilde{g}\in{X}$ the function ${g}=\Gamma[\tilde{g}]$ satisfies in almost all points $(t,x)$
the linear differential equation
\begin{equation}
\frac{\partial {g}}{\partial t}(t,x)=\left\{\begin{aligned}
 &-\frac {\Dconst}{\tilde N(t)}{g}(t,x)\tilde{g}(t,t-x)\quad && x<t,\\
 &0&& x>t
\end{aligned}
\right.\label{eq:aux}
\end{equation}
with $\tilde{N}(t)=\int_0^t{}\tilde{g}(t,x)dx$ and initial data $g(1,x)=g_\ini(x)$ . Our
main result can be summarized as follows.
\begin{theorem}
\label{Theo.Existence.Mild.Solutions}
For given initial data $g_\ini$ with \eqref{Ass.Initial.Data} and any $T>0$ there exists a
unique mild solution $u\in{X}$ to \eqref{Eqn:Model} that satisfies \eqref{Eqn:MildSolution} with
$\Gamma[g]=g$.
\end{theorem}

%
%
\subsection{Existence proof for fixed points of $\Gamma$ }
%
To employ the Contraction Mapping Principle we identify a subset $S\subset X$ that is invariant
under $\Gamma$ and a metric such that $X$ is complete, $S$ is closed, and $\Gamma$ is a
contraction. In what follows $S$ is given by
$$
S=\left\{\tilde g\in X\left|\ \forall (t,x)\in\Omega:\quad 0\leq\tilde g(t,x)\leq f(x),\
\frac12\leq\int_0^ty\tilde g(t,y)dy\leq\frac32\right.\right\},
$$
where $f:[0,T]\rightarrow\R^+$ will be identified below. This set $S$ is invariant under
$\Gamma$
provided that each $\tilde g\in{S}$ satisfies
\begin{align}
0\leq\Gamma[\tilde g](t,x)\leq f(x)  	&\qquad\forall (t,x)\in\Omega,  \label{invarcond1}\\
\frac12\leq\int_0^tx\Gamma[\tilde g](t,x)dx\leq\frac32 &\qquad\forall t\in[1,T].
\label{invarcond2}
\end{align}
Towards \eqref{invarcond1} we estimate
$$\tilde N(t)=\int_0^t\tilde g(t,x)dx\geq \frac 1T\int_0^tx\tilde g(t,x)dx \geq \frac 1{2T},$$
and this implies $1/\tilde N(t)\leq2T$. Moreover,
\begin{enumerate}
\item[i)] if $x\leq 1$, then
  $$
   \Gamma[\tilde g](t,x)=g_\ini(x)\exp{\left(-{\Dconst}\int_1^t\frac{\tilde g(s,s-x)}{\tilde
N(s)}ds\right)}
                                       \leq g_\ini(x)\leq h_0:=
\norm{g_\ini}_\infty,                                       
  $$
\item[ii)] if $t<x$, then $\Gamma[\tilde g](t,x)=0\leq f(x)$,
\item[iii)] if $1\leq x\leq t$, then
  $$
   \Gamma[\tilde g](t,x)=B[\tilde g](x)\exp{\left(-{\Dconst}\int_x^t\frac{\tilde
g(s,s-x)}{\tilde N(s)}ds\right)}
                                       \leq B[\tilde g](x)\leq T{\Dconst}\int_0^xf(y)f(x-y)dy.
  $$
\end{enumerate}
Therefore, $f$ satisfies (\ref{invarcond1}) provided that
\begin{equation}
\int_0^xf(y)f(x-y)dy\leq\frac{f(x)}{T\Dconst}\label{ineq:cond1}\;\;\;\text{for all}\;\;\;x\geq1,
\qquad\text{and}\qquad
f(x)\geq h_0\;\;\;\text{for all}\;\;\;x\leq1.
\end{equation}
\begin{lemma}
\label{lemma1}%
There exists a constant $D>0$ such that $f(x)=h_0e^{Dx^3}$ satisfies
(\ref{ineq:cond1}).
\end{lemma}
\begin{proof}
There is nothing to show for $x\leq1$. For $x\geq1$, condition (\ref{ineq:cond1}) can be
rewritten as
$$ \begin{aligned}
 \frac1{T\Dconst}&\geq\int_0^xe^{D(y^3+(x-y)^3-x^3)}dy
                   &=\int_0^xe^{-3Dxy(x-y)}dy.
\end{aligned}$$
Now suppose that $A\in[0,1/2]$ is given. Then for all $x\geq 1$ we have $x\geq x-\frac
Ax\geq\frac Ax\geq0$, and this implies
$$ \begin{aligned}
 \int_0^xe^{-3Dxy(x+y)}dy&=\left[\int_0^{\frac Ax}+\int_{\frac Ax}^{x-\frac Ax}+\int_{x-\frac
Ax}^x\right]e^{-3Dxy(x+y)}dy\\
                         &\leq\frac{2A}x+xe^{-3Dx\frac Ax(x+\frac Ax)}
                         \leq\frac{2A}x+xe^{-3DAx}.
\end{aligned}
$$
For the first inequality we have estimated the integrand in the first and third part of the
integral by $1$, and in the second part by its value on the boundaries, which can be done
since the integrand is a convex function. The second inequality then follows as  $A$, $D$,
and $x$ are non-negative.
\par
We now choose $A$ with $A<(4T\Dconst)^{-1}$, and this implies  $\frac{2A}x<\frac1{2T\Dconst}$
for all $x\geq 1$. Next we choose D large enough so that
$e^{-3DA}<\frac1{2T\Dconst}$ and $3DA>1$.
Since the function $xe^{-3DAx}$ is decreasing for $x\geq 1$ we find
$xe^{-3DAx}<\frac1{2T\Dconst}$
for all $x\geq 1$. Hence our choice of $A$ and $D$ guarantees (\ref{ineq:cond1}).
\end{proof}
We now define
$$M_T:=\max\{f(t):t\in[0,T]\}.$$
Notice that $\tilde g\in S$ implies $\tilde{g}(t,x)\leq M_T$ for all $(t,x)\in\Omega$.
\begin{lemma}
 We can choose $T>0$ such that (\ref{invarcond2}) is satisfied for all $\tilde{c}\in{S}$.
\end{lemma}
\begin{proof}
For $\tilde g\in S$ we have
\begin{align*}
\frac{\partial}{\partial{t}}\Gamma[\tilde g](t,x)
=%
-\frac{\Dconst}{\tilde N(t)}\Gamma[\tilde g](t,x)\tilde g(t,t-x)
\end{align*}
because $\Gamma[\tilde g]$ is piecewise continuously differentiable in $t$ and satisfies
(\ref{eq:aux}). The continuity properties of $\Gamma[\tilde g]$ imply that 
$\int_0^tx\Gamma[\tilde g](t,x)dx$ is differentiable in time, and we estimate
\begin{eqnarray*}
\left|\frac d{dt}\int_0^tx\Gamma[\tilde g](t,x)dx\right|&=&\left|t\Gamma[\tilde
g](t,t)-\int_0^t\frac{\partial}{\partial t}\Gamma[\tilde g](t,x)dx\right|\\
&=&\left|\frac {\Dconst}{\tilde N(t)}\int_0^tx\tilde g(t,x)\tilde g(t,t-x)dx - \frac
{\Dconst}{\tilde N(t)}\int_0^tx\Gamma[\tilde g](t,x)\tilde g(t,t-x)dx\right|\\
                 &\leq&\frac {\Dconst}{\tilde N(t)}\int_0^tx\tilde g(t,t-x)\left|\tilde
g-\Gamma[\tilde g]\right|(t,x)dx\\
                 &\leq&4{\Dconst}T^3M_T^2.
\end{eqnarray*}
For $t=1$ we have $\int_0^1x\Gamma[\tilde g](1,x)dx=\int_0^1x{g_\ini}(x)dx=1$, so using the
above
bound we can choose $T$ such that $\int_0^t\Gamma[\tilde g](t,x)dx$ stays between $1/2$ and
$3/2$ for all $t\in[1,T]$.
\end{proof}
We have now shown that $S$ is invariant under $\Gamma$, and that there
exists a constant $M_T$ such that $\tilde
c(t,x)\leq M_T$ holds for all $\tilde g\in S$ and all $(t,x)\in\Omega$.
\par
In the next step we construct a norm for $X$ such that $\Gamma$ is a contraction on $S$. To
this end, we define
\begin{align*}
\norm{\tilde g}_1:=\sup_{(t,x)\in\Omega} |\tilde g(t,x)|,\qquad
\norm{\tilde g}_2:=\sup_{t\in[1,T]}\int_0^T|\tilde g(t,x)|dx,
\end{align*}
and derive an estimate for $\norm{\Gamma[\tilde g_1]-\Gamma[\tilde g_2]}_{1,2}$ in terms
of $\norm{\tilde g_1-\tilde g_2}_{1,2}$. Afterwards we show that $\Gamma$
is a contraction with respect to some linear combination of these norms.
\begin{lemma}
\label{Lem1}
There exists a constant $L>0$ such that for any $\tilde g_1,\tilde g_2\in S$
\begin{eqnarray*}
\norm{\Gamma[\tilde g_1]-\Gamma[\tilde g_2]}_1&\leq& L(T-1)\norm{\tilde g_1-\tilde
c_2}_1+L\norm{\tilde g_1-\tilde g_2}_2,\\
\norm{\Gamma[\tilde g_1]-\Gamma[\tilde g_2]}_2&\leq& L(T-1)\norm{\tilde g_1-\tilde
c_2}_1+L(T-1)\norm{\tilde g_1-\tilde g_2}_2.
\end{eqnarray*}
\end{lemma}
\begin{proof}
Let $(t,x)\in\Omega$. Since $|e^a-e^b|\leq|a-b|$ for all $a,b>0$, we obtain the following
inequalities. For $x<1$ we find that
\begin{equation*}
 |\Gamma[\tilde g_1](t,x)-\Gamma[\tilde g_2](t,x)|      \leq
 h_0{\Dconst}\int_1^t\left|\frac{\tilde g_1(s,s-x)}{\tilde N_1(s)}-\frac{\tilde
g_2(s,s-x)}{\tilde N_2(s)}\right|ds
\end{equation*}
\begin{equation}
\begin{aligned}
       \leq&\  	h_0{\Dconst}\int_1^t\left\{2T|\tilde g_1(s,s-x)-\tilde
g_2(s,s-x)|+M_T\frac{|\tilde N_1(s)-\tilde N_2(s)|}{\tilde N_1(s)\tilde N_2(s)}\right\}ds\\
       \leq&\  	(2h_0T{\Dconst}\norm{\tilde g_1-\tilde g_2}_1+4T^2M_Th_0{\Dconst}\norm{\tilde
g_1-\tilde g_2}_2)\int_1^t ds\\
       \leq&\  	2h_0T{\Dconst}(T-1)\norm{\tilde g_1-\tilde
g_2}_1+4T^2M_Th_0{\Dconst}(T-1)\norm{\tilde g_1-\tilde g_2}_2\label{ineq:norm1}.
\end{aligned}
\end{equation}
%
Now let $x \geq 1$, and set $E[\tilde g](t,x):=\exp(-{\Dconst}\int_{x}^t\frac{\tilde
g(s,s-x)}{\tilde N(s)}ds)$. Then,
\begin{equation*}
 |\Gamma[\tilde g_1](t,x)-\Gamma[\tilde g_2](t,x)|=
 |B[\tilde g_1](x)E[\tilde g_1](t,x)-B[\tilde g_2](x)E[\tilde g_2](t,x)|
\end{equation*}
\begin{equation*}
\begin{aligned}
            \leq&\  	|B[\tilde g_1](x)-B[\tilde g_2](x)|\cdot|E[\tilde g_1](t,x)|+\\
            \qquad&\ 	|B[\tilde g_2](x)|\cdot|E[\tilde g_1](t,x)-E[\tilde g_2](t,x)|\\
            \leq&\ 	|B[\tilde g_1](x)-B[\tilde g_2](x)|+T^2{\Dconst}M_T^2|E[\tilde
g_1](t,x)-E[\tilde g_2](t,x)|.
\end{aligned}
\end{equation*}
We treat the last two summands separately. Analogously to (\ref{ineq:norm1}) we  estimate
\[
|B[\tilde g_1](x)-B[\tilde g_2](x)|\leq 	
(2TM_T{\Dconst}+2{\Dconst}M^2_TT^3)\norm{\tilde g_1-\tilde g_2}_2.
\]
On the other hand, we can estimate $|E[\tilde g_1](t,x)-E[\tilde g_2](t,x)|$ by
\[
\begin{split}
             	\int_x^t\big\{2T{\Dconst}|\tilde g_1(s,s-x)-\tilde
g_2(s,s-x)|&+4M_TT^2{\Dconst}|\tilde N_1(s)-\tilde N_2(s)|\big\}ds\\
            \leq&\ 	2T{\Dconst}(T-1)\norm{\tilde g_1-\tilde g_2}_1+4T^3{\Dconst}\norm{\tilde
g_1-\tilde g_2}_2.
\end{split}
\]
Combining these results we find a constant $L'$ that depends polynomially on
${\Dconst}, T, M_T$ such that
\begin{equation}
\begin{aligned}
|\Gamma[\tilde g_1](t,x)-\Gamma[\tilde g_2](t,x)|  &\leq  L'(T-1) \norm{\tilde g_1-\tilde
g_2}_1+L'(T-1)\norm{\tilde g_1-\tilde g_2}_2, && x<1,\\
|\Gamma[\tilde g_1](t,x)-\Gamma[\tilde g_2](t,x)|  &\leq  L'(T-1) \norm{\tilde g_1-\tilde
g_2}_1+L'\norm{\tilde g_1-\tilde g_2}_2, && x\geq 1.\\
\end{aligned}\label{ineq:norm2}
\end{equation}
To derive the bounds for the second norm we split the interval $[0,T]=[0,1]\cup[1,t]\cup[t,T]$,
and using (\ref{ineq:norm2}) we find
\begin{equation}
  \int_0^T|\Gamma[\tilde g_1](t,x)-\Gamma[\tilde g_2](t,x)|dx\leq L'T(T-1)\norm{\tilde
g_1-\tilde g_2}_1+2L'(T-1)\norm{\tilde g_1-\tilde g_2}_2.\label{ineq:norm3}
\end{equation}
With $L:=\max\{TL',2L'\}$  we then derive from (\ref{ineq:norm2}), (\ref{ineq:norm3}) that
\begin{equation*}
\begin{aligned}
	|\Gamma[\tilde g_1](t,x)-\Gamma[\tilde g_2](t,x)|  &\leq  L(T-1)\norm{\tilde g_1-\tilde
g_2}_1+L\norm{\tilde g_1-\tilde g_2}_2,\\
	\int_0^T|\Gamma[\tilde g_1](t,x)-\Gamma[\tilde g_2](t,x)|dx  &\leq  L(T-1)\norm{\tilde
g_1-\tilde g_2}_1+L(T-1)\norm{\tilde g_1-\tilde g_2}_2.
\end{aligned}
\end{equation*}
The assertions now follow by taking the supremum in the above inequalities.
\end{proof}
Now let $\beta\geq 2L$, where $L$ is as in the proof of Lemma \ref{Lem1}, and define a norm on
$X$ by \begin{align*}
       \norm{\cdot}:=\norm{\cdot}_1+\beta\norm{\cdot}_2.
       \end{align*}
For any $\tilde g_1,\tilde g_2\in S$ we then have
$$\norm{\Gamma[\tilde g_1]-\Gamma[\tilde g_2]}  \leq L(1+\beta)(T-1)\norm{\tilde g_1-\tilde
g_2}_1+L(1+\beta(T-1))\norm{\tilde g_1-\tilde g_2}_2,$$ and it is possible to choose
$T$ such 
that $\beta(T-1)<1/2$ and $L(1+\beta)(T-1)<1/2$. Hence $L(1+\beta(T-1))<3\beta/4$, and this
gives
$$
\norm{\Gamma[\tilde g_1]-\Gamma[\tilde g_2]}  \leq
\frac 12\norm{\tilde g_1-\tilde g_2}_1+\frac 34\beta\norm{\tilde g_1-\tilde g_2}_2
\leq \frac 34\norm{\tilde g_1-\tilde g_2}.
$$
$X$ equipped with $\norm{\cdot}$ is a Banach Space and $S$ is a closed and bounded subset, so
the Banach Fixed
Point Theorem guarantees that $\Gamma$ has a unique fixed point $g\in S$. By construction,
this fixed point solves the differential equation (\ref{Eqn:Model})$_1$ for $t\leq{T}$.
Moreover, since $g(t,t)=B[g](t)$, it also satisfies condition (\ref{eq:gen3}), which is
equivalent to (\ref{Eqn:Model})$_3$.
\bigpar
It remains to prove that there exists a solution for all $1<T<\infty$. This can be done by
standard methods because (i) for each $T>1$ there exists a constant $D=D(T)>0$ as in
Lemma \ref{lemma1}, and (ii) each solution satisfies 
$\int_0^tx\Gamma[\tilde g](t,x)dx=1$ for all $1\leq{t}\leq{T}$. 
%
%
\section{Self-similar solutions} \label{Sec:SSS}
%
%
In this section we describe self-similar solutions to \eqref{Eqn:Model}. These satisfy
\begin{equation*}
g(t,x)=\left\{
\begin{aligned}	
&\alpha(t)G\left(\frac{x}{t}\right)\quad 	&& x\leq
t,\\
&0&& x >t,
\end{aligned}
\right.%
\end{equation*}
where $G:[0,1]\rightarrow \R_0^+$ is a continuously differentiable function, and the mass
constraint \eqref{Eqn:Model}$_3$ requires
 \begin{align*}
	1=\int_0^t xg(t,x)dx=\int_0^t x\alpha
(t)G\left(\frac{x}{t}\right)dx=t^2\alpha(t)\int_0^1 yG(y)dy.
\end{align*}
This means $t^2\alpha (t)$ is a positive constant. By rescaling $\alpha$ and $G$ we can
ensure that this constant is $1$, so each mass preserving self-similar solution takes the
form $g(t,x)=t^{-2}G(x/t)$ with $\int_0^1 yG(y)dy=1$. Using this relation in
\eqref{Eqn:Model}, and substituting $y=x/t$, we get
\begin{equation}
2G(y)+yG'(y)=\frac{\Dconst}{N}G(y)G(1-y),\quad
\label{eq:ss1} 	N=\int_0^1 G(y)dy.
\end{equation}
We first consider a simplified problem
\begin{equation}
2G(y)+yG'(y)=DG(y)G(1-y),\label{eq:ss:aux}	
\end{equation}
where $D>0$ is an arbitrary constant and we do not impose the mass constraint on $G$. At
first, we notice that each solution $G$ to (\ref{eq:ss:aux}) provides a solution $\tilde{G}$
to (\ref{eq:ss1}) via
\begin{equation}
\Dconst=D\int_0^1  G(x)dx,\qquad \tilde{G}(y)=\frac{ G(y)}{\int_0^1 x
G(x)dx}\label{eq:ss:trans}.
\end{equation}
We also observe that (\ref{eq:ss:aux}) is invariant under the scaling
$G\rightsquigarrow\lambda{G}$, $D\rightsquigarrow\lambda^{-1}{D}$. Consequently, in order to
characterize the solution set of \eqref{eq:ss1}, we have to investigate (\ref{eq:ss:aux}) for
only one value of $D$, and then consider how the corresponding solutions transform under
(\ref{eq:ss:trans}).
\begin{lemma}
For any given $D>0$ and $G(1/2)>0$ there exists a unique positive solution of
(\ref{eq:ss:aux}) on $(0,1)$.
\end{lemma}
\begin{proof}
We multiply (\ref{eq:ss:aux}) by $y$ and substitute $F(y)=G(y)y^2$ to obtain
\begin{equation}
F'(y)=\frac{F(y)F(1-y)}{y(1-y)^2}.\label{eq:ss3}
\end{equation}
Next, we decompose $F$ into its odd and even parts with respect to $1/2$ by setting
\begin{eqnarray*}
F(y)  &=&F_e(y)+F_o(y),\\
F(1-y)&=&F_e(y)-F_o(y),
\end{eqnarray*}
so (\ref{eq:ss3}) transforms into 
\begin{eqnarray}
\begin{aligned}F'_e(y)&=\frac{D(F_e(y)^2-F_o(y)^2)}{y^2(1-y)^2}\left(y-\frac{1}{2}\right),\\
F'_o(y)&=\frac{D(F_e(y)^2-F_o(y)^2)}{2y^2(1-y)^2}.
\end{aligned}\label{eq:ss4}
\end{eqnarray}
Since (\ref{eq:ss4}) is locally Lipschitz for $y\in(0,1)$, the local existence and uniqueness of
solutions to the initial value problems is granted. In particular, for given values $F_o(1/2)=0$
and $F_e(1/2)>0$ we find the smallest $a\in[0,1/2)$ such that there exists a unique solution
to (\ref{eq:ss4}) on $(a,1-a)$. The function $F(y)=F_e(y)+F_o(y)$ then solves (\ref{eq:ss3})
on $(a,1-a)$, and satisfies
\begin{align*}
F(y)=F(1/2)\exp{\left(\int_{1/2}^y \frac{DF(1-z)}{z(1-z)^2}dz\right)}.
\end{align*}
In particular, $F$ is positive and by (\ref{eq:ss3}) it is also increasing on
$(a,1-a)$. If $a=0$ then we are done. Otherwise we know that $DF(1-y)y^{-1}(1-y)^{-2}$ is
bounded for $y\in [1/2,1-a)$, which means that $F$ does not blow-up at $1-a$. Thus, 
$F_e$, $F_o$ are also bounded on $(a,1-a)$ and we can extend the solution to $[a,1-a]$. Due to
the local existence result, we can then extend the solution to an interval $(b,1-b)$ with $b\in
[0,a)$, which contradicts the minimality of $a$.  Hence we have $a=0$, and the proof is
complete.
\end{proof}
\begin{figure}[ht!]
\centering{%
\includegraphics[width=.32\textwidth]{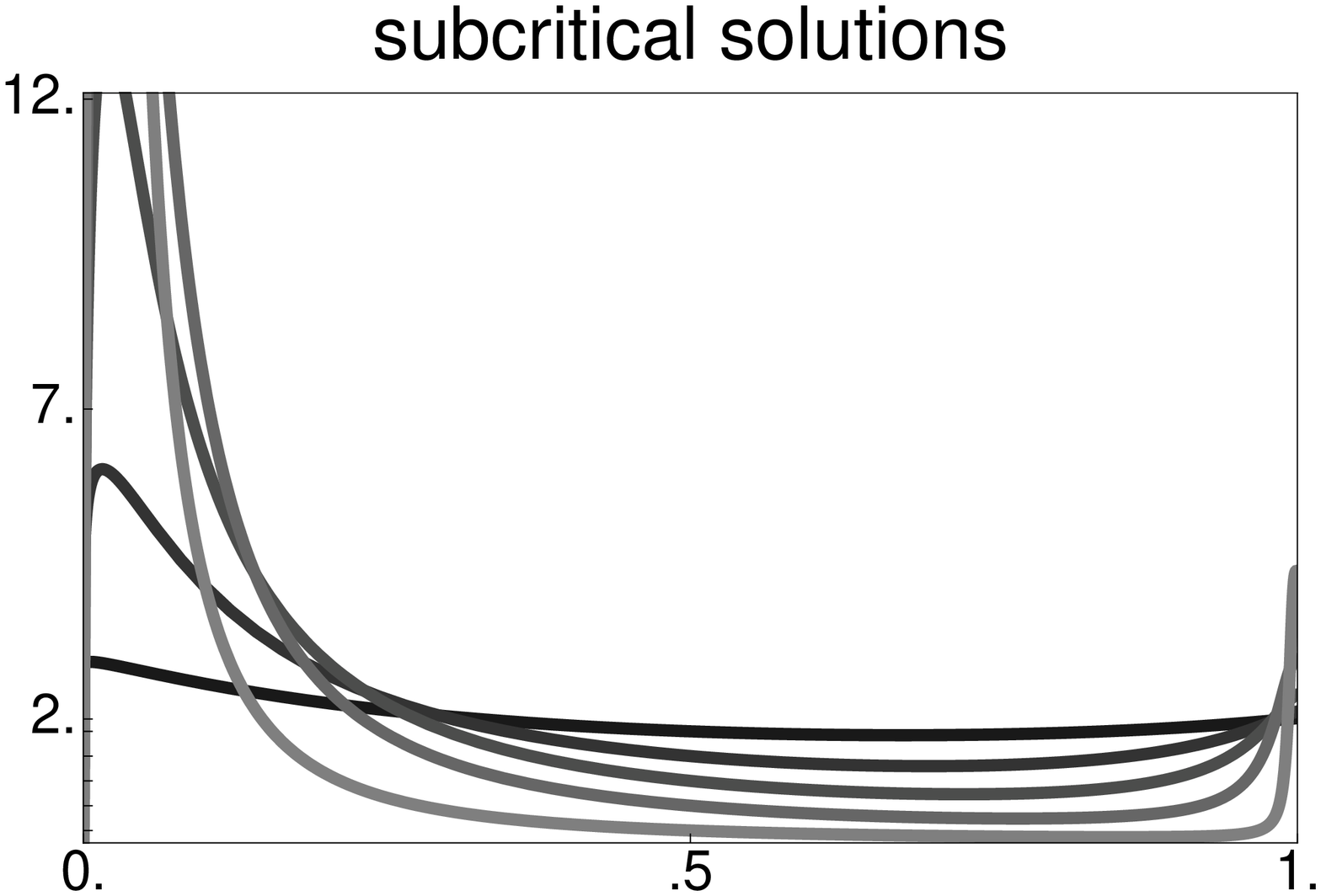}%
\includegraphics[width=.32\textwidth]{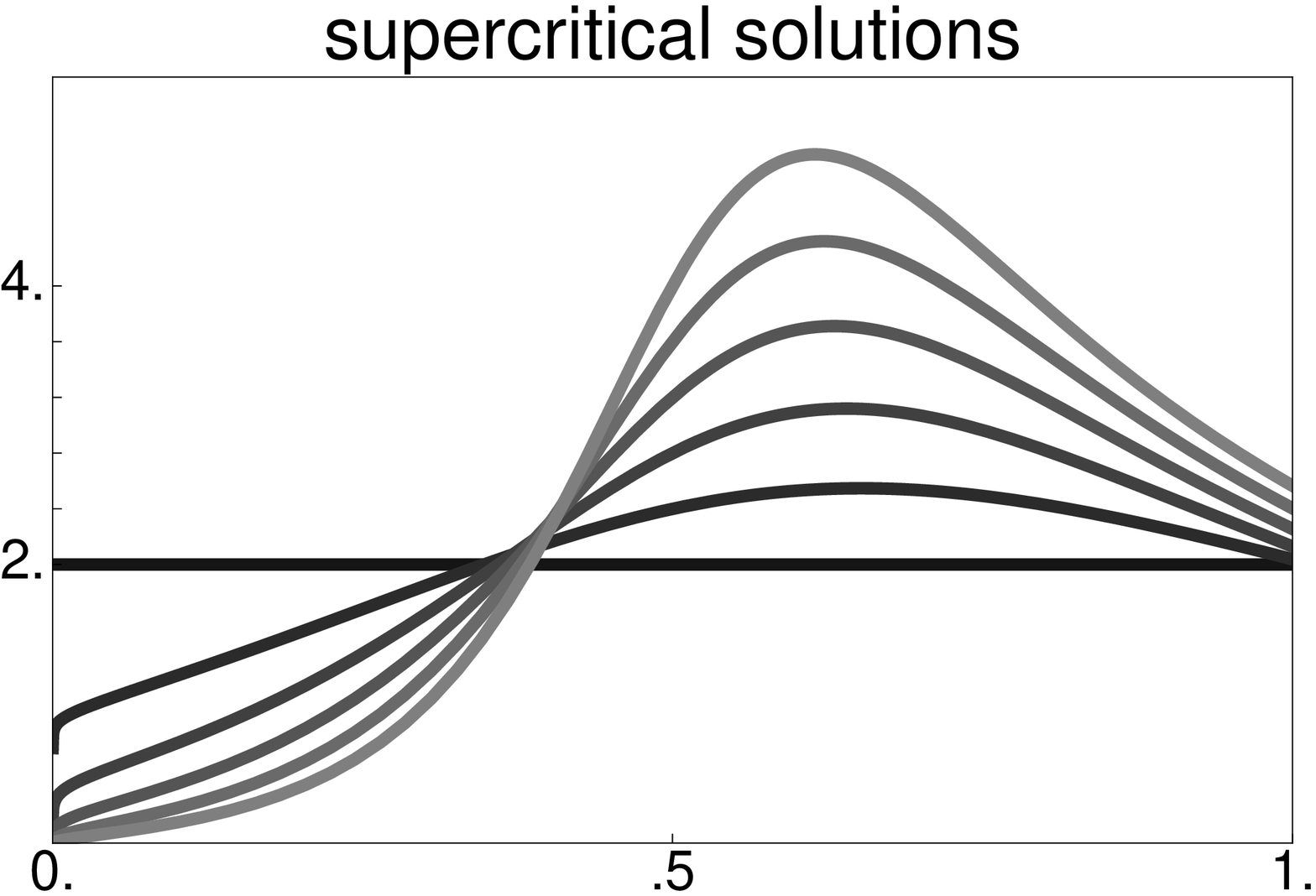}%
\includegraphics[width=.32\textwidth]{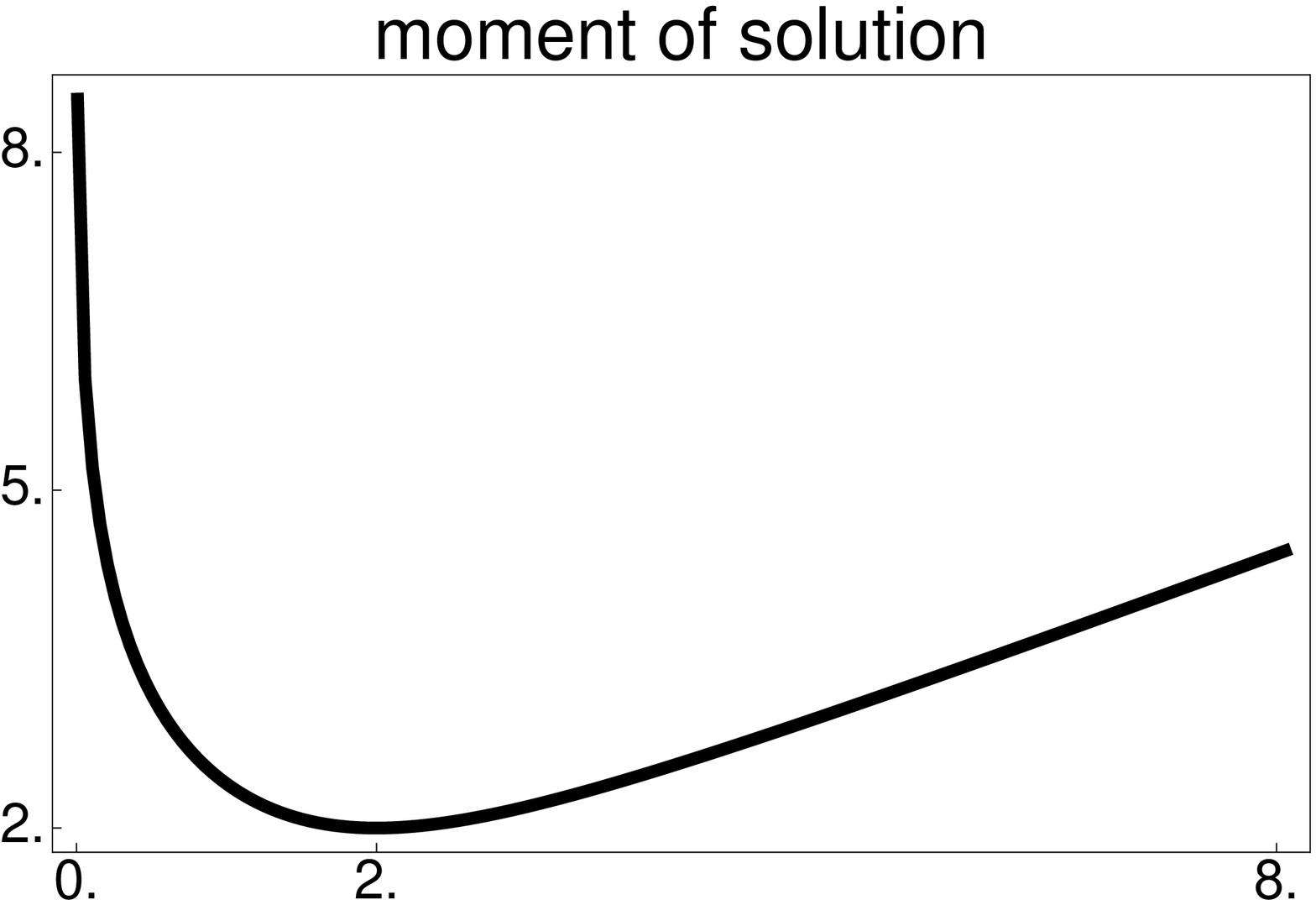}%
}%
\caption{%
Solutions $G(y)$ to \eqref{eq:ss:aux} with $D=1$ for $y\in[0,\,1]$ for
different values of $G(1/2)<2$ (left) and $G(1/2)>2$ (center);
moment $N(G)=\int_0^1G(y)d{y}$ in dependence on $G(1/2)$ (right).
}%
\label{fig:sss1}
\end{figure}
We next discuss some numerical ODE simulations that illustrate how the solutions of
\eqref{eq:ss:aux} with $D=1$ depend on the value of
$G(1/2)$. For
$G(1/2)=2$ we get the trivial solution $G\equiv 2$. From now on, we refer to solutions with
$G(1/2)>2$ as supercritical and to those with $G(1/2)<2$ as subcritical. These two types
behave rather differently, see Figure \ref{fig:sss1}, which shows the solutions for
$G(1/2)\in\{0.2,\,0.6,\,1.0,\,1.4,\,1.8\}$ and
$G(1/2)\in\{2.0,\,2.4,\,2.8,\,3.2,\,3.6,\,4.0\}$. Our numerical results indicate that each
supercritical solution has precisely one local maximum between $1/2$ and $1$ but no local
minimum. On the other hand, a subcritical solution has two local maxima close to $0$ and $1$,
and a local minimum between $1/2$ and $1$, compare Figure
\ref{fig:sss2}, where 'variation near $1$' refers to $G(y)-G(1)$.
\begin{figure}[ht!]
\centering{%
\includegraphics[width=.305\textwidth]{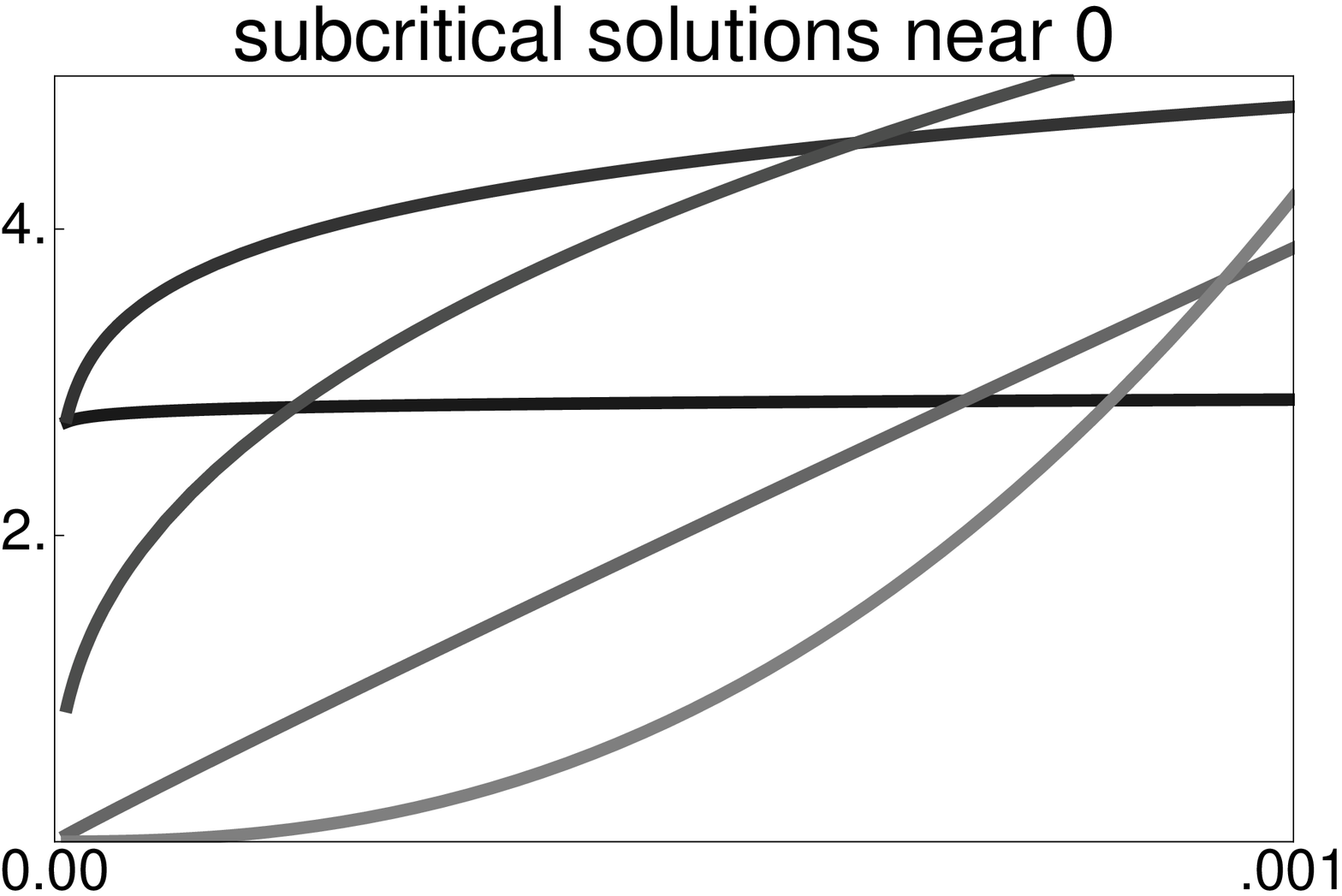}%
\includegraphics[width=.305\textwidth]{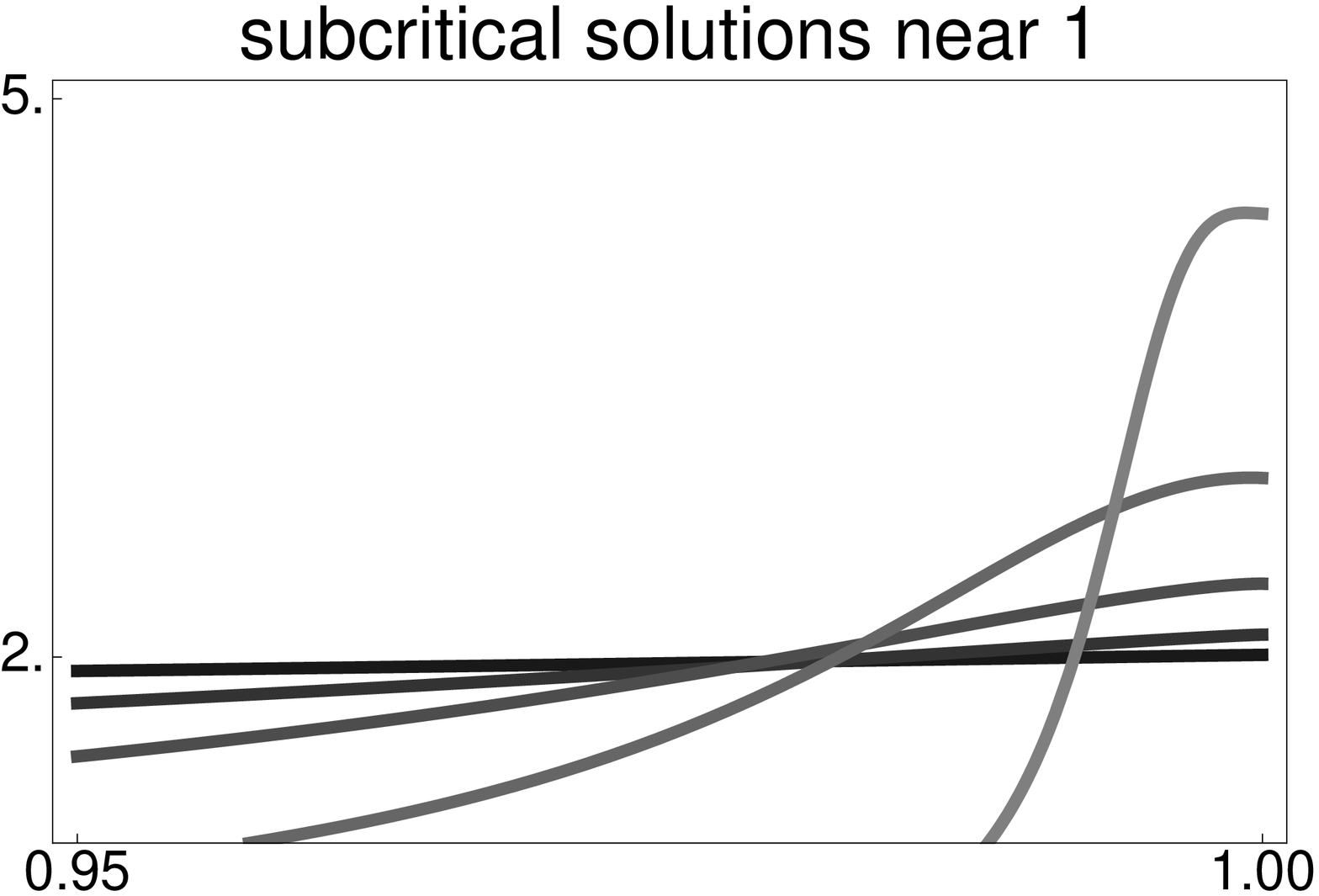}%
\includegraphics[width=.33\textwidth]{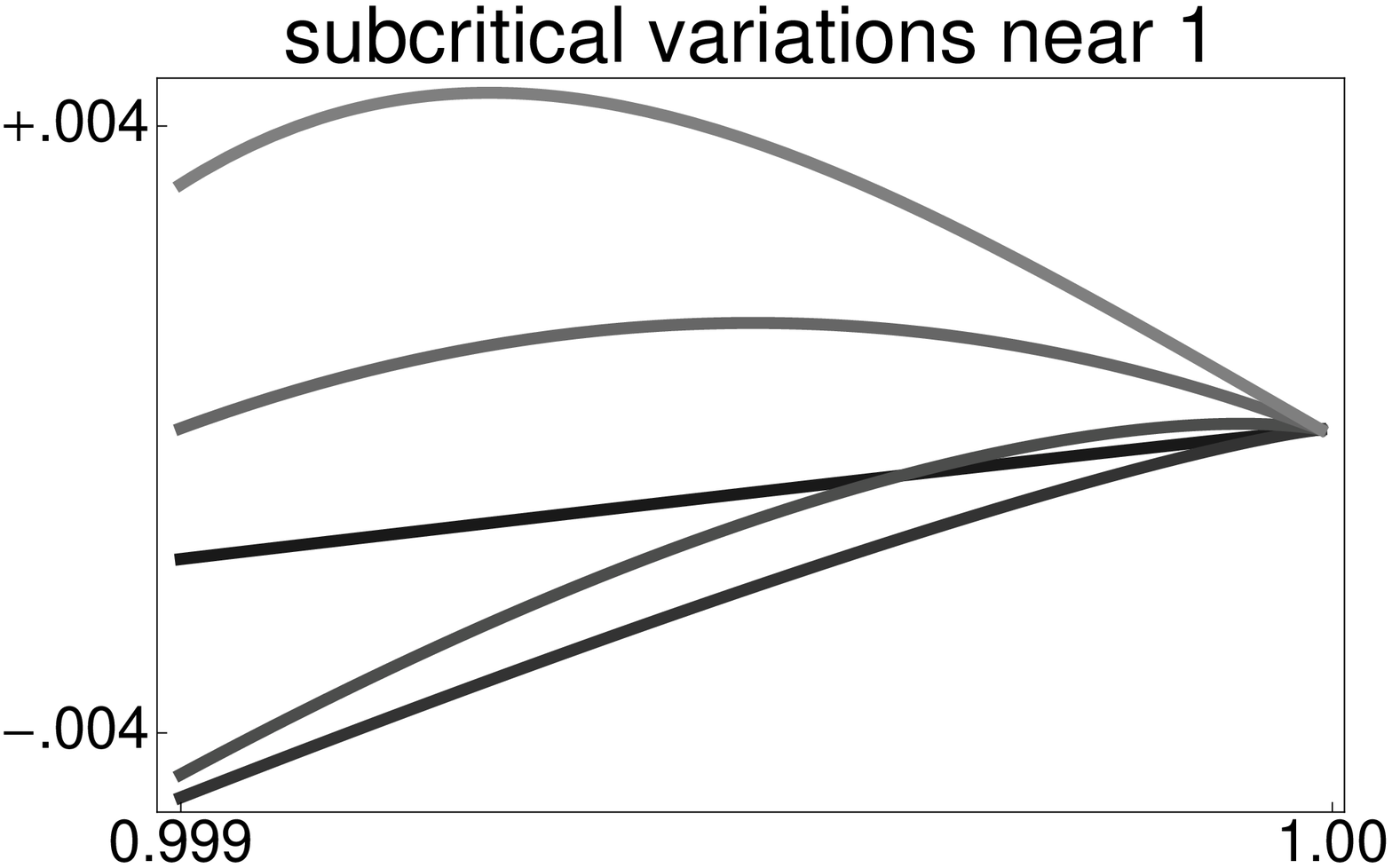}%
}%
\caption{Subcritical solutions from Figure \ref{fig:sss1} for $y\approx0$ and $y\approx1$.}
\label{fig:sss2}
\end{figure}
\par
Figure \ref{fig:sss1} also shows how $\int_0^1 G(y)dy$ depends on $G(1/2)$. Based on this, we
conjecture that for every ${\Dconst}>2$ there exist two solutions to (\ref{eq:ss1}), one
having a subcritical and the other having a supercritical shape. For ${\Dconst}=2$ there
is a trivial self-similar solution. For ${\Dconst}<2$ there seem to be no self-similar
solution. To support this conjecture, we now prove that there is no self-similar solution for
${\Dconst}\leq 1$: Integrating \eqref{eq:ss1} we find $\lim_{z\rightarrow 1} G(z)=N$, so each
solution to (\ref{eq:ss1}) satisfies $G(y)\sim y^{-2+{\Dconst}}$ as $y\rightarrow 0$,
which contradicts $\int_0^1G(y)dy<\infty$. For $1<\Dconst<2$, however, this argument does not apply 
but our numerical results indicate there is still no singular solution.
\newcommand{\NB}{M_\mathrm{b}}
\section{Long-time behavior}\label{Sec:LTB}
%
%
To investigate the long time behaviour of solutions to (\ref{Eqn:Model}) by numerical
simulations we derive a discrete "box model" which follows naturally from the physical
interpretation. We assume the number of initial boxes $\NB$ is sufficiently large (we choose
$\NB=200$ for all simulations) and consider the discrete times $t=\eps{j}$
with $j\in\mathbb{N}$ and $\eps=1/\NB$. Moreover, we denote by
$G(j,i)$ the number of particles with size $x\in(\eps{i}-\eps,\eps{i})$ at time
$t=1+\eps{j}$, that means 
\begin{align*}
G(j,i)={\eps}^{-1}\int_{\eps(i-1)}^{\eps{i}}g(1+\eps{j},x)dx 
\end{align*}
for all integers $j\geq0$ and $i=1,\ldots,\NB+j$. The disrete analogue to the evolution
equation (\ref{Eqn:Model}) is then given by
\begin{align*}
\frac{G(j+1,i)-G(j,i)}{\eps}=\left\{\begin{aligned}
  &-\frac {\Dconst}{{N}(j)}G(j,i)G(j,\NB+j+1-i)\quad && 0< i\leq j+\NB,\\
  &0&& i>j+\NB+1,
\end{aligned}
\right.
\end{align*}
where $N(j)=\eps\sum_{i=1}^{j+\NB} G(j,i)$ and the value of $G(j,j+\NB+1)$ is
determined by the conservation of mass, i.e.,
\begin{align*}
G(j,j+\NB+1)=\frac{\eps}{2}\sum_{i=1}^{j+\NB}\frac {\Dconst}{{N}(j)}G(j,i)G(j,\NB+j+1-i).
\end{align*}
%
%
\bigpar
We now present our numerical results for three different values of ${\Dconst}$ and  three
different sets of initial data. More precisely, we consider ${\Dconst}=1, 2, 3$ and assume that the
initial data have Gaussian distributions with dispersion 0.3 and center at either $0.25$,
$0.5$, or $0.75$.
\par
For ${\Dconst}>2$ we expect convergence to one of the self-similar solutions. Numerical
simulations suggest that the solution converges as $t\rightarrow\infty$ to the supercritical
self-similar solution that corresponds to ${\Dconst}$. The same happens if the initial data
is very close to the subcritical self-similar solution, and thus we can conclude that subcritical
solutions are unstable. We failed to find a rigorous proof for this assertion, but the numerical
evidence is strong. In Figure \ref{fig1}, we plot the scaled distributions after 0, 200, 1000
and 25000 steps (0, 1000, 25000 and 150000 if the center is at 0.25) for the three  sets of 
initial data described above. The corresponding times are given by 0, 1, 5, and 125 (0, 5,
125, and 750). Along with these smooth curves we dotted the graph of the self-similar
solution for ${\Dconst}=3$. As we see, the convergence is slowest when the center of the
initial data is at 0.25.
\begin{figure}[ht!]
\includegraphics[width=0.3\textwidth]{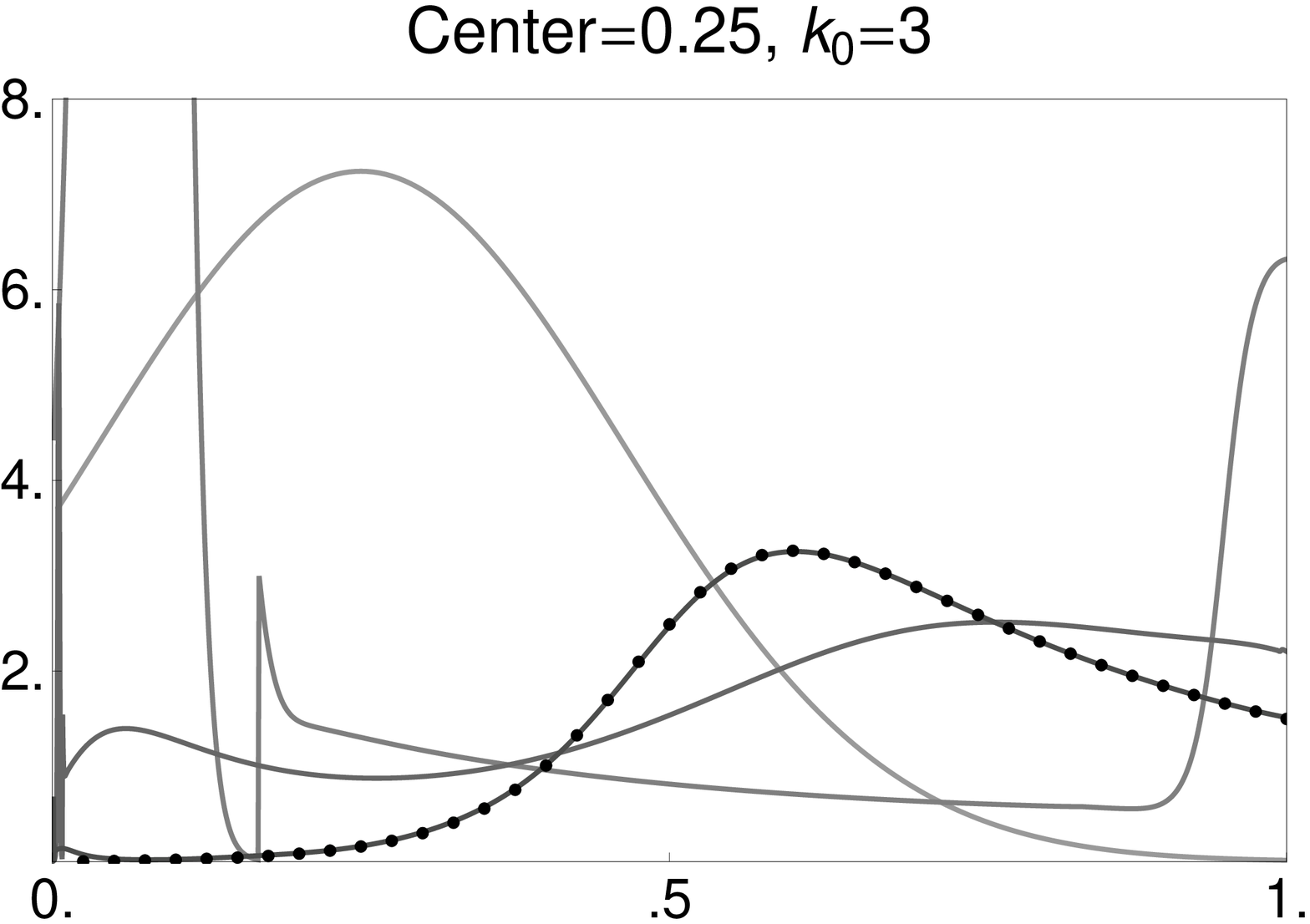}
\includegraphics[width=0.3\textwidth]{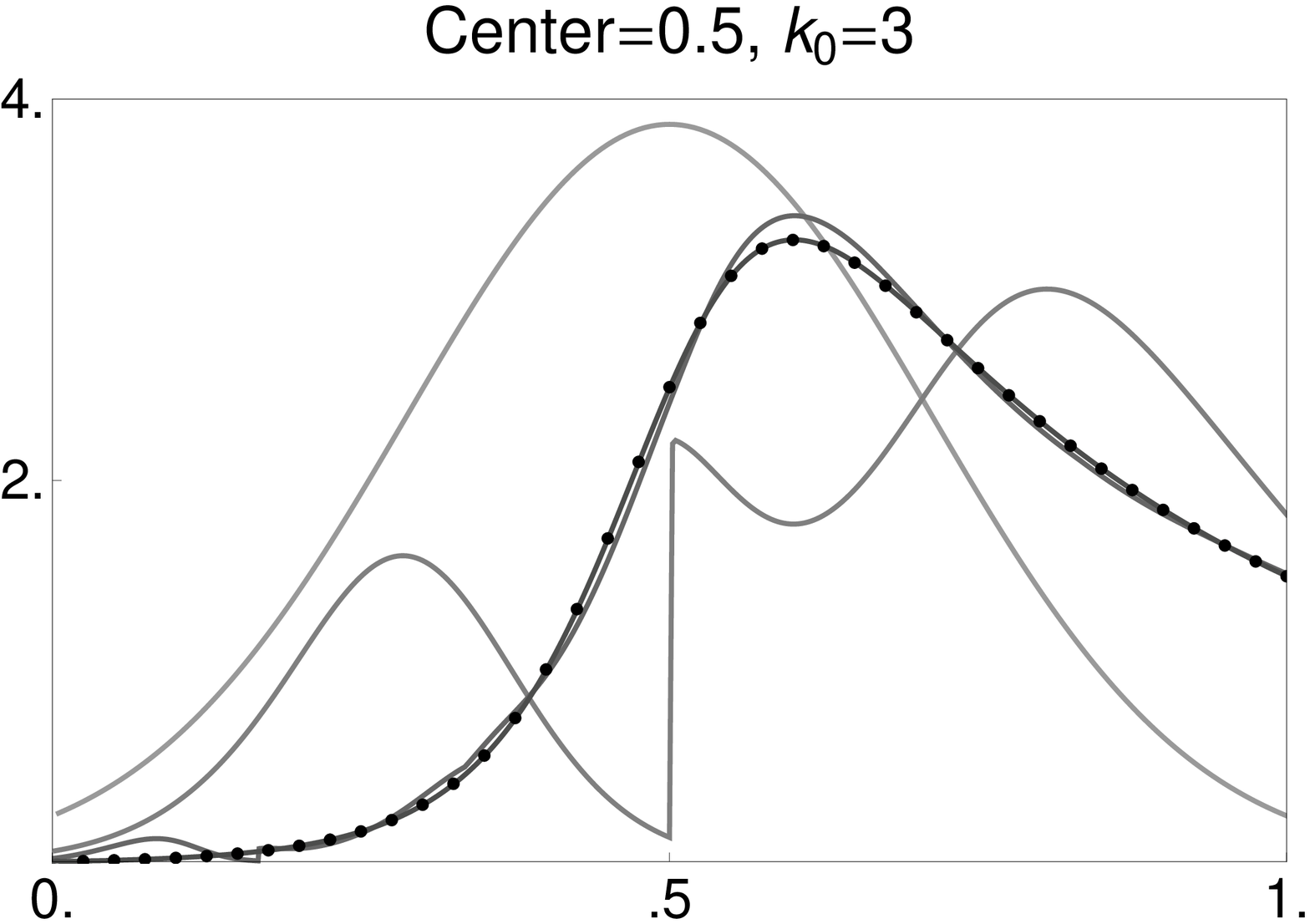}
\includegraphics[width=0.3\textwidth]{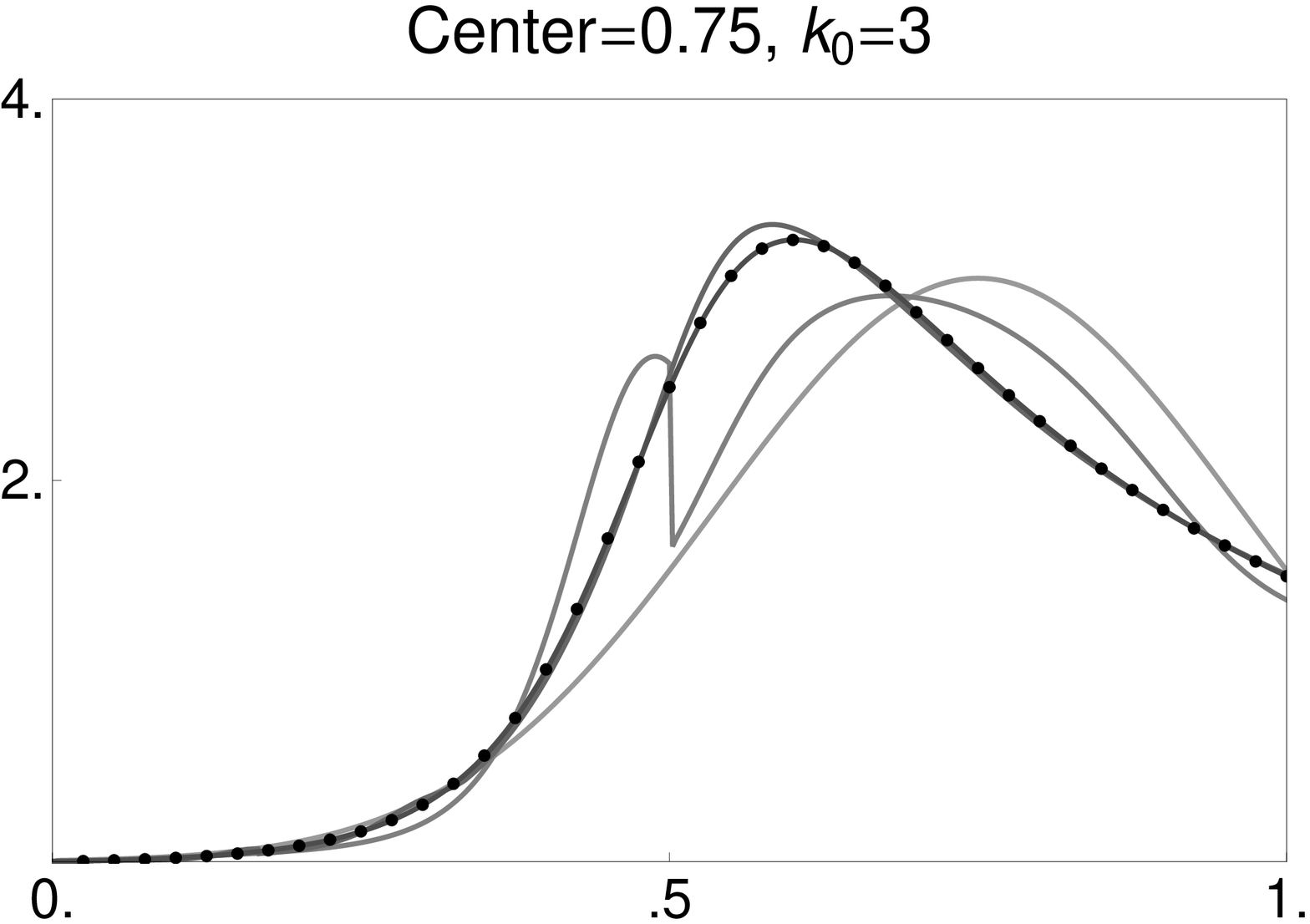}
\caption{Convergence to self-similar form for ${\Dconst}=3$}
\label{fig1}
\end{figure}
\par
For ${\Dconst}<2$ we observe a completely different behaviour. We do not see any convergence
in the self-similar scaling. On the contrary, the solutions apparently converge  pointwise in
the original variables to a limit that depends on the initial data. We will prove the
corresponding statement rigorously in Proposition \ref{T.main} for ${\Dconst}<1/3$. Figure
\ref{fig2} shows the unscaled distributions for ${\Dconst}=1$, the initial data is the same
as before. The five smooth curves represents numerical distributions after 0, 200, 1000, 5000
and 25000 steps.
\begin{figure}[ht!]
\includegraphics[width=0.3\textwidth]{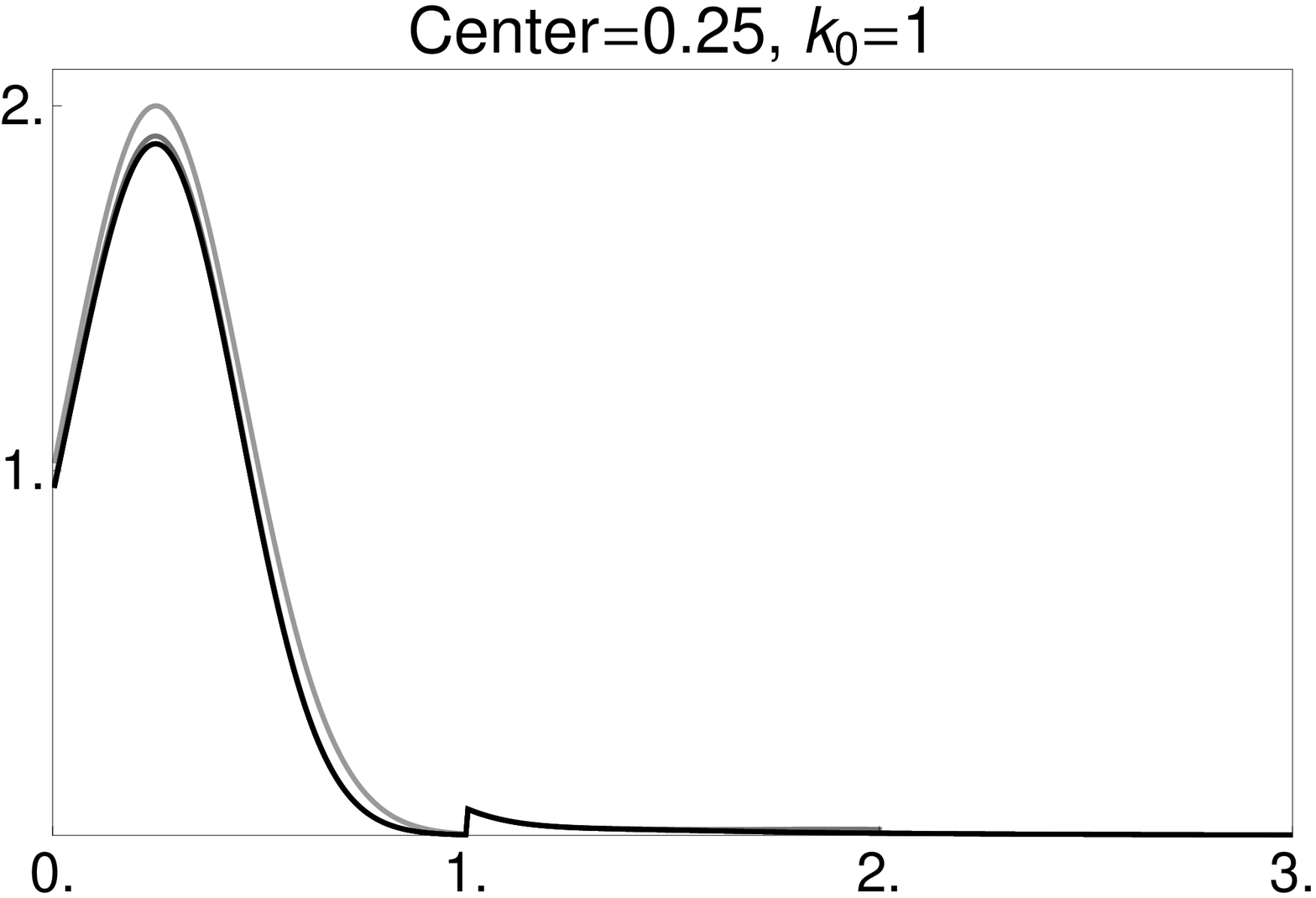}
\includegraphics[width=0.3\textwidth]{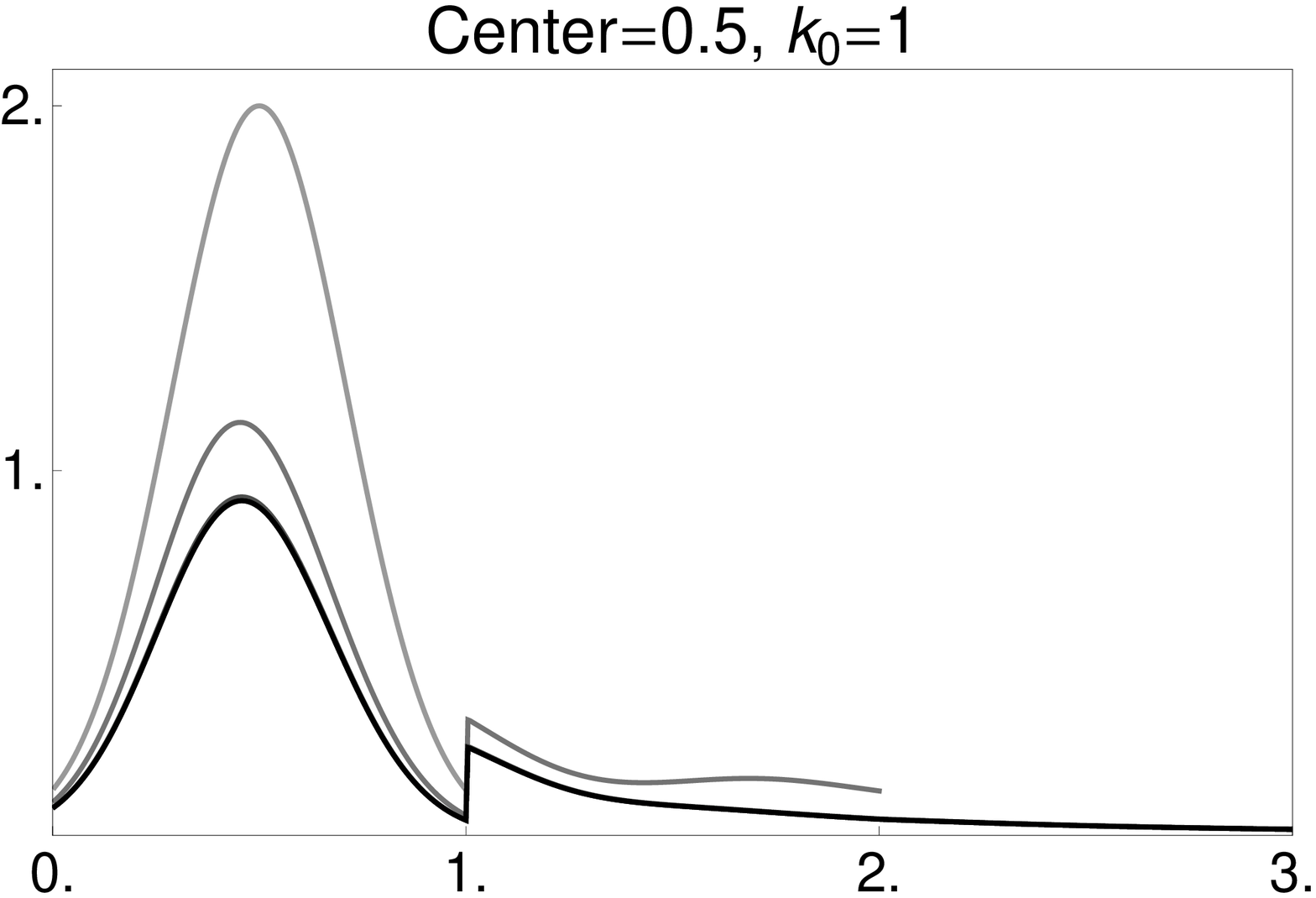}
\includegraphics[width=0.3\textwidth]{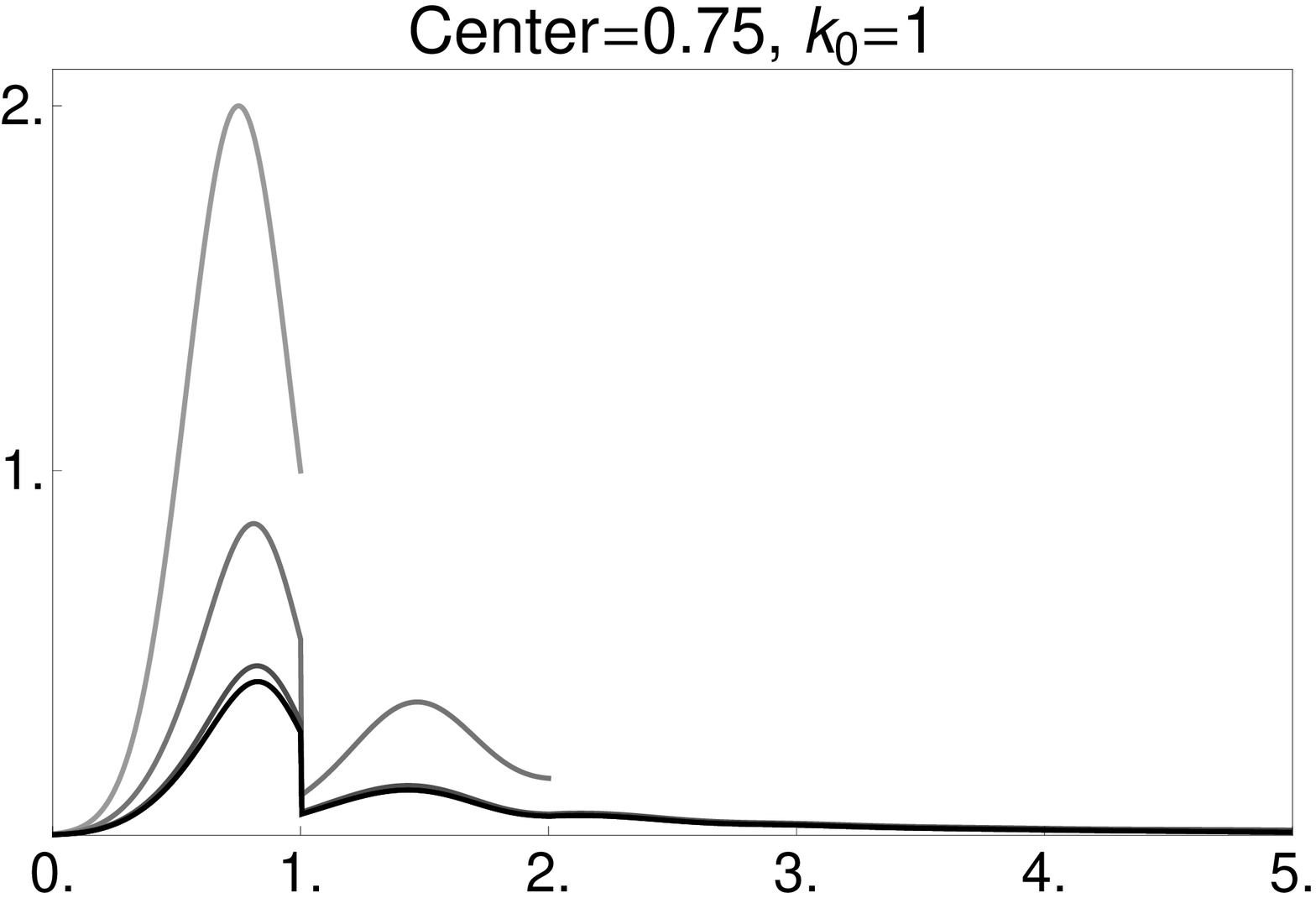}
\caption{Convergence in the original variables for ${\Dconst}=1$}
\label{fig2}
\end{figure}
\par
For ${\Dconst}=2$ we expect convergence to the trivial self-similar solution but the
numerical simulations do not support this assertion. We can conclude that either the long
time behaviour is more complicated or the convergence is extremely slow. Figure
\ref{fig3}
contains the distributions after 0, 1000, 5000, 25000 and 150000 steps and the dotted
plot of the self-similar solution. In the first picture we observe a behavior similar to that
for ${\Dconst}<2$, that means the mass is cumulated near the origin. It will possibly eventually
disappear but this was not the case after any number of steps that we were able to simulate.
This happens also for ${\Dconst}>2$ if the initial data is more
cummulated. For example, if we choose dispersion equal to 0.2 then 150000 steps 
is not enough to see the convergence for ${\Dconst}=3$ and center at $0.25$. This
observation is not surprising because if the initial data is supported on $[0,1/2)$, then the
evolution cannot start at all.
\begin{figure}
\includegraphics[width=0.3\textwidth]{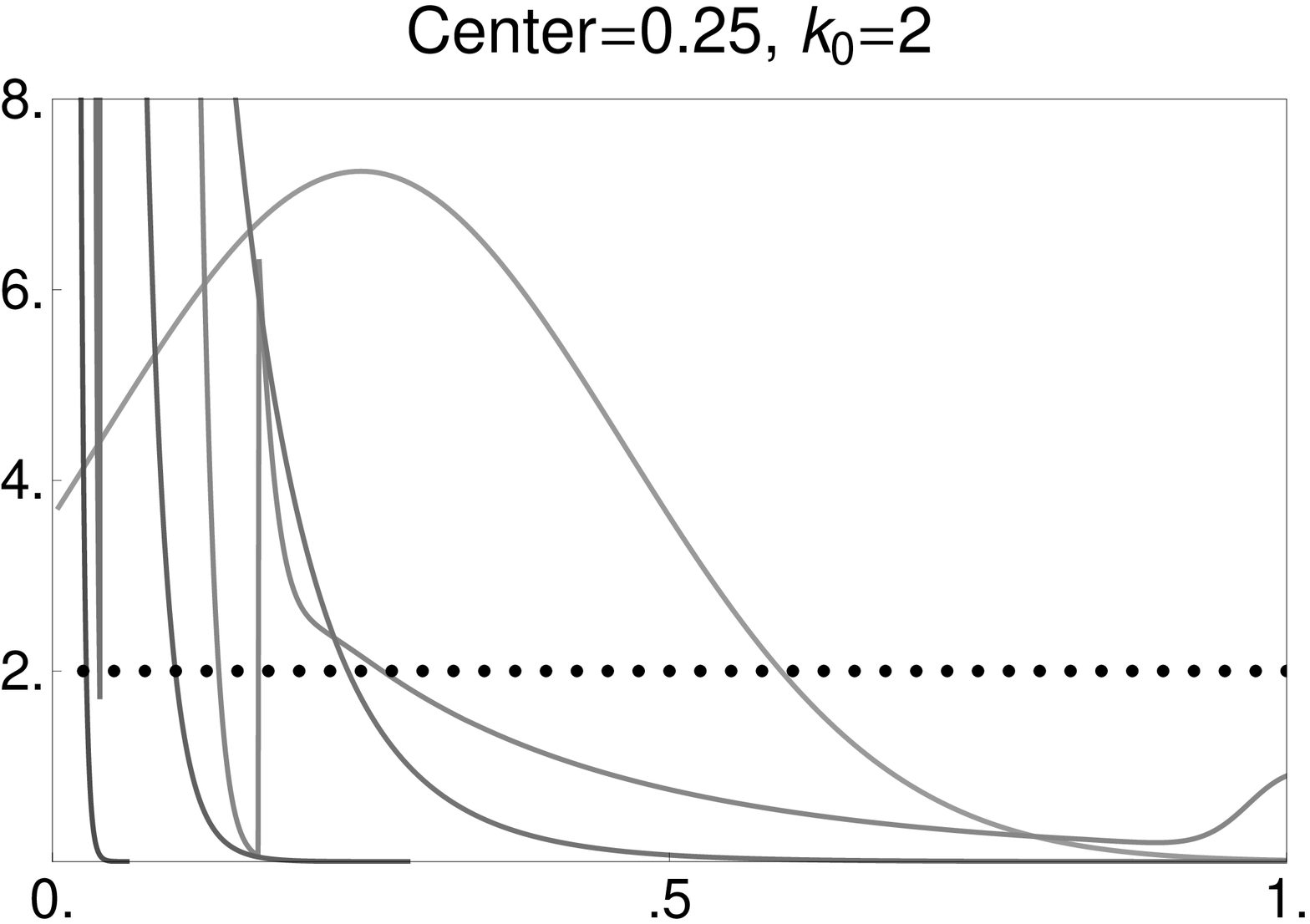}
\includegraphics[width=0.3\textwidth]{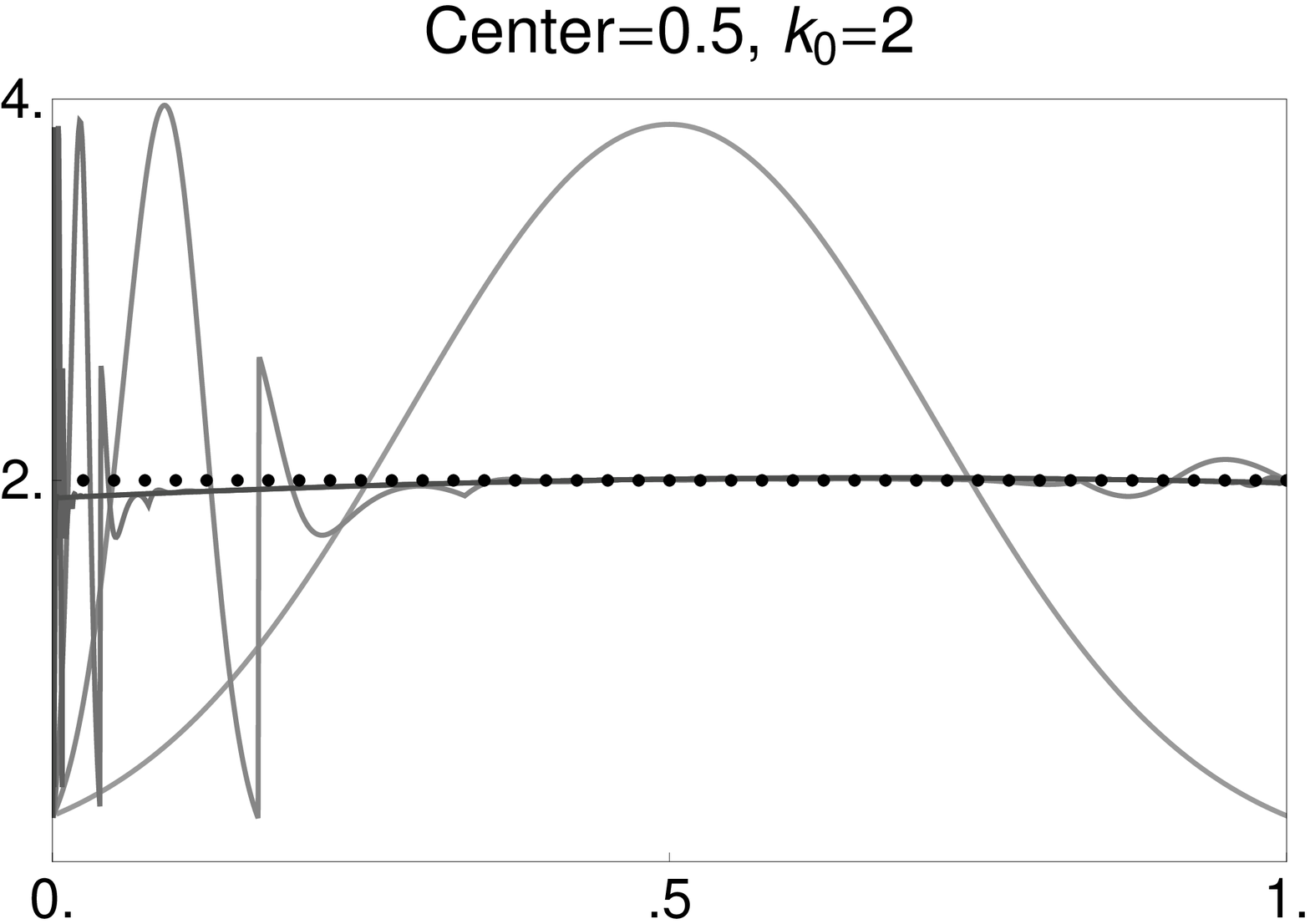}
\includegraphics[width=0.3\textwidth]{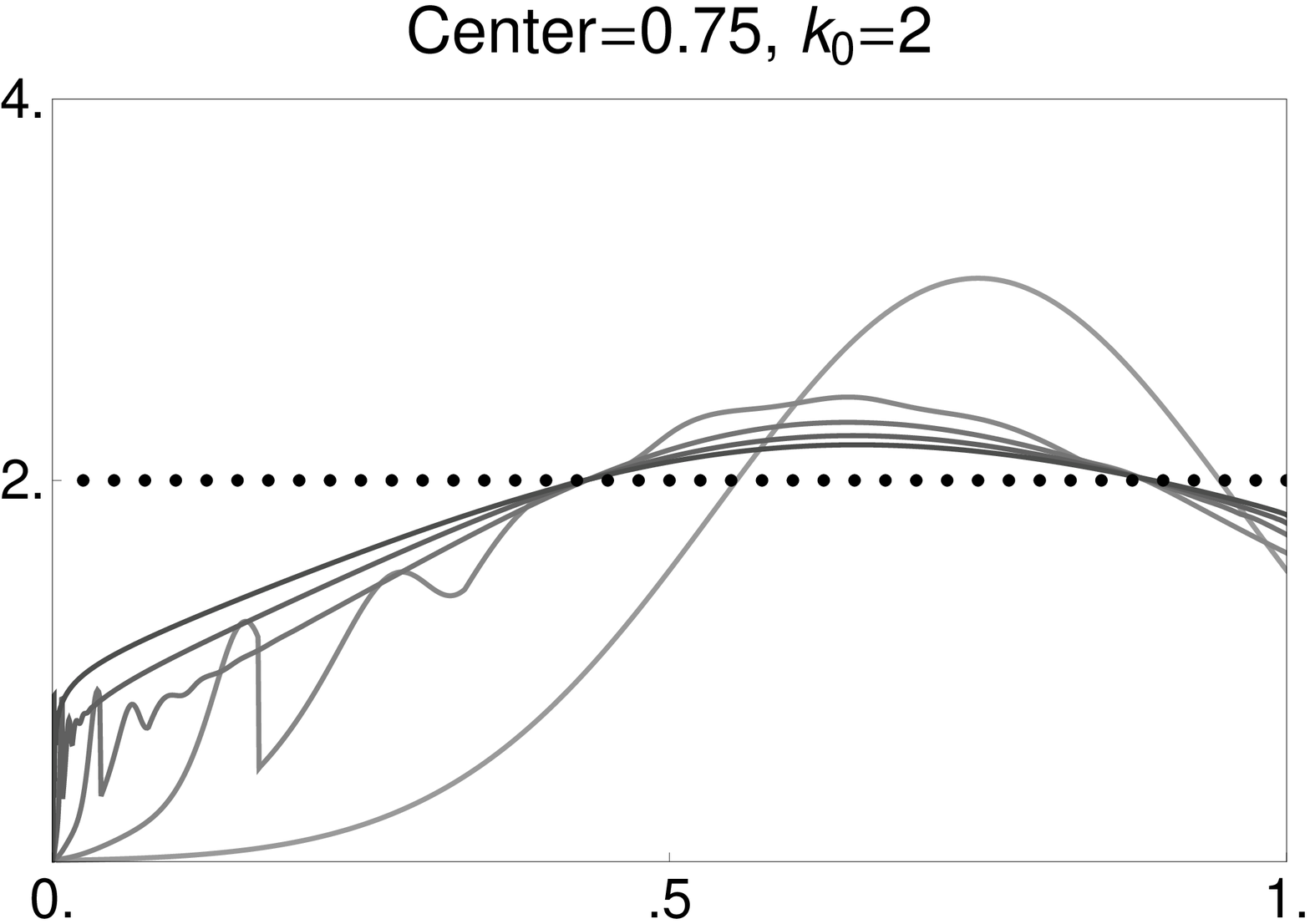}
\caption{Evolution in self-similar variables for $d=2$}
\label{fig3}
\end{figure}

We conclude this section with a rigorous proof of a theorem supporting the conjecture that for
${\Dconst}<2$ there exists a limit function
\begin{equation}
\label{eq:ltb}
\lim_{t\rightarrow \infty} g(t,x)=:g_{\infty}(x)>0.
\end{equation}
Notice that \eqref{eq:ltb} implies that the rescaled function $t^2g(t,yt):[0,1]\rightarrow\R$
converges, in the sense of probability measures, to a Dirac measure with positive mass.
\begin{proposition}
\label{T.main}
\begin{enumerate}
\item[i)] %
If there exists $c_0>0$ such that $N(t)\geq c_0$ for all $t\geq 1$, then 
the relation (\ref{eq:ltb}) is satisfied.
\item[ii)]
If ${\Dconst}<1/3$ then $N(t)>N(1)/2$ for all $t\geq 1$.
\end{enumerate}
\end{proposition}
\begin{proof}
i) Suppose that $N(t)\geq c_0$. From \eqref{Eqn:Model} and \eqref{Eqn:MildSolution}
with $\tilde{g}=g$ we infer that 
\begin{align*}
g(t,x)=g(max(1,x),x)\exp{\left(-{\Dconst}\int_{max(1,x)}^t
\frac{g(s,s-x)}{N(s)}ds\right)}. 
\end{align*}
By assumption, we also have
\begin{align*}
\int_{max(1,x)}^t \frac{g(s,s-x)}{N(s)}ds \leq \frac{1}{c_0}\int_{max(1,x)}^t
g(s,s-x)ds<\infty,
\end{align*}
so $g(t,x)$ is bounded from below. Since it is also decreasing in $t$ for all $x$, 
there exists $g_\infty(x)$ as in \eqref{eq:ltb}.
\par 
ii) Due to (\ref{Eqn:Model}), we have $-2N(t)N'(t)={\Dconst}\int_0^t g(t,x)g(t,t-x)dx$ and
by integration we obtain
\begin{equation}
\begin{aligned}
N(1)^2-N(t)^2 	&= 	{\Dconst}\int_1^t\int_0^s g(s,x)g(s,s-x)dxds\\
		&=	{\Dconst}\int_0^t\int_{max(1,x)}^t g(s,x)g(s,s-x)dsdx\\
		&\leq	{\Dconst}\int_0^t g(max(1,x),x)\int_{max(1,x)}^t g(s,s-x)dsdx\\
		&\leq	{\Dconst}\int_0^t g(max(1,x),x)\int_0^t g(max(1,s),s)dsdx\\
		&\leq	{\Dconst}(2N(1)-N(t))^2,
\end{aligned}\label{eq:ltb2}
\end{equation}
where the last inequality holds since
$$\int_0^t g(max(1,s),s)ds=\int_0^1 g(1,s)ds+\int_1^t g(s,s)ds=N(1)+(N(1)-N(t)).$$
Now suppose that there exists $t>1$ such that $N(t)= N(1)/2$. By (\ref{eq:ltb2}) we get ${\Dconst}\geq 1/3$
and using the continuity of $N$, we conclude that
${\Dconst}<1/3$ implies $N(t)>N(1)/2$ for all $t$.
\end{proof}
\paragraph*{Acknowledgment}
This work was supported by the EPSRC Science and Innovation award to the
Oxford Centre for Nonlinear PDE (EP/E035027/1). Ondrej Bud\'a\v c also gratefully
acknowledges support through the SPP Foundation and by the Slovak Research and Development
Agency under the contract No. APVV-0414-07.

\end{document}